\documentclass[a4paper,12pt, reqno]{amsart}
\usepackage{amsfonts}
\usepackage{amsmath, array, bigints}
\usepackage{amssymb}
\usepackage{mathrsfs, mathtools}
 \usepackage{graphicx}
\usepackage{enumerate}

\usepackage[usenames]{color}
\makeatletter
\newcommand{\setword}[2]{%
  \phantomsection
  #1\def\@currentlabel{\unexpanded{#1}}\label{#2}%
}
\makeatother
\newtheorem{thm}{Theorem}[section]
\newtheorem{cor}[thm]{Corollary}
\newtheorem{lem}[thm]{Lemma}
\newtheorem{prop}[thm]{Proposition}

\numberwithin{equation}{section}
\usepackage[english]{babel} 
\usepackage{blindtext}

\theoremstyle{definition}
\newtheorem{definition}[thm]{Definition}
\newtheorem{rem}[thm]{Remark}

\usepackage[colorlinks]{hyperref}

\begin{document}

\allowdisplaybreaks 
\def\dist{{\operatorname{dist}}}

\renewcommand{\d}{\:\! \mathrm{d}}


 \title[Infinitely many solutions for nonlinear superposition operators]{Superlinear problems involving nonlinear superposition operators of mixed fractional order}

\author[Souvik Bhowmick, Sekhar Ghosh, and Vishvesh Kumar]{Souvik Bhowmick, Sekhar Ghosh and Vishvesh Kumar}

\address[Souvik Bhowmick]{Department of Mathematics, National Institute of Technology Calicut, Kozhikode, Kerala, India - 673601}
\email{souvikbhowmick2912@gmail.com / souvik\_p230197ma@nitc.ac.in}

\address[ Sekhar Ghosh]{Department of Mathematics, National Institute of Technology Calicut, Kozhikode, Kerala, India - 673601}
\email{sekharghosh1234@gmail.com / sekharghosh@nitc.ac.in}
\address[Vishvesh Kumar]{Department of Mathematics: Analysis, Logic and Discrete Mathematics, Ghent University, Ghent, Belgium}
\email{vishveshmishra@gmail.com / vishvesh.kumar@ugent.be}
\date{}

\begin{abstract} In this work, we study a class of elliptic problems involving nonlinear superpositions of fractional operators of the form  
\[
A_{\mu,p}u := \int_{[0,1]} (-\Delta)_{p}^{s} u \, d\mu(s),
\]
where $\mu$ is a signed measure on $[0,1]$, coupled with nonlinearities of superlinear type. Our analysis covers a variety of superlinear growth assumptions, beginning with the classical Ambrosetti--Rabinowitz condition. Within this framework, we construct a suitable variational setting and apply the Fountain Theorem to establish the existence of infinitely many weak solutions. The results obtained are novel even in the special cases of superpositions of fractional $p$-Laplacians, or combinations of the fractional $p$-Laplacian with the $p$-Laplacian. More generally, our approach applies to finite sums of fractional $p$-Laplacians with different orders, as well as to operators in which fractional Laplacians appear with ``wrong'' signs.  A distinctive contribution of the paper lies in providing a unified variational framework that systematically accommodates this broad class of operators.

\end{abstract}

\keywords{Nonlocal Superposition Operators, Superlinear problems, Cerami condition, Fountain theorem.\\
\textit{2020 Mathematics Subject Classification: } 35M12, 35J60, 35R11, 35A15, 35S15, 35J60}

\maketitle

\section{Introduction and main results}
The purpose of this article is to investigate nonlinear elliptic problems driven by nonlinear superposition operators of mixed fractional order. In recent years, considerable attention has been devoted to the existence theory for elliptic equations involving such operators. These superposition operators unify and generalize several fundamental cases, including the $p$-Laplacian, the fractional $p$-Laplacian, their combinations (i.e., mixed local and nonlocal operators), as well as their differences \cite{DPSV}. From a purely mathematical perspective, the interplay of these operators gives rise to a loss of scale invariance, which in turn may introduce unexpected difficulties. The situation becomes particularly delicate since these operators are no longer positive definite. Consequently, the bilinear form naturally associated with them does not yield a scalar product or a norm, the variational spectrum may contain negative eigenvalues, and even maximum principles can fail \cite{DPSV2025}.

At the same time, superposition operators of mixed fractional type are not only of theoretical interest but also appear in several applied contexts. They play an important role in the modeling of anomalous diffusion, population dynamics, and various problems in mathematical biology, including those related to Gaussian processes and L\'evy flights. For further discussions and related developments, we refer the reader to \cite{DPSV25, DPSV2025,DV21}.  
To be precise, we study the following superlinear problem involving the superposition nonlinear operator of mixed fractional order:
\begin{align}\label{p}
A_{\mu, p} u &= f(x,u) ~\text{in}~\Omega  \nonumber\\
 u & =0 ~\text{in}~ \mathbb{R}^N \setminus \Omega,
 \end{align}
where $\Omega$ is a bounded open subset of $\mathbb{R}^{N}$, $ 1<p<\frac{N}{s_\sharp}$ and $1<p<q<p^*_{s_\sharp}:=\frac{Np}{N-ps_\sharp}$ with $s_\sharp$ defined by \eqref{mu 4} below. The (nonlinear) superposition operator (of fractional order)  $A_{\mu, p}$ is defined as follows:
\begin{equation}\label{sup op}
    A_{\mu, p} u:=\int_{[0,1]}(-\Delta_p)^{s} u \d \mu(s),
\end{equation} 
where $\mu$ is a signed measure given by
\begin{equation} \label{mu}
    \mu:=\mu^{+}-\mu^{-}. 
\end{equation}
Consider two nonnegative finite Borel measures $\mu^{+}$ and $\mu^{-}$ on $[0,1]$, which are assumed to satisfy the following conditions:

\begin{equation}\label{mu 1}
    \mu^{+}([\bar{s}, 1])>0,
\end{equation}
\begin{equation}\label{mu 2}
    \mu^{-}|_{[\bar{s}, 1]}=0,
\end{equation}
and
\begin{equation}\label{mu 3}
\mu^{-}([0, \bar{s}]) \leq \kappa \mu^{+}([\bar{s}, 1]) ,
\end{equation}
for some $\bar{s} \in(0,1]$ and $\kappa \geq 0.$ In fact, from \eqref{mu 1}, there exists another (fractional) exponent $s_{\sharp} \in[\bar{s}, 1]$ such that 
\begin{equation}\label{mu 4}
\mu^{+}\left(\left[s_{\sharp}, 1\right]\right)>0. 
\end{equation}

The superposition operator defined in \eqref{sup op} exhibits remarkable versatility. It can represent the negative $p$-Laplacian $-\Delta_{p} u$ when $\mu$ is the Dirac measure concentrated at 1, the fractional $p$-Laplacian $(-\Delta)_p^{s}$ when $\mu$ is concentrated at a fractional power $s \in (0, 1)$, and even a “mixed-order” operator of the form $-\Delta_p + \epsilon (-\Delta)_p^{s}$, with $\epsilon \in (0,1]$, when $\mu$ is the sum of two Dirac measures, $\delta_1 + \epsilon \delta_s$, $s \in (0,1)$, among other possibilities.  Interestingly, some components of these operators may contribute with the “wrong sign”; however, as long as a dominant term, typically corresponding to a higher-order fractional operator, provides sufficient control, the overall behaviour remains well-posed.   

Problem \eqref{p} admits a natural variational structure, with its solutions corresponding to the critical points of the associated Euler–Lagrange functional. This naturally leads to the question of whether the well-established topological and variational methods, originally developed in the classical setting, can be effectively extended to this unified nonlinear and nonlocal framework. We now define the notion of a weak solution to the problem \eqref{p}.

\begin{definition}\label{weaksol MP}
    We say $u \in X_{p}(\Omega)$ is a weak solution to the problem \eqref{p}, if
\begin{align*}
& \int_{[0,1]}\left(\iint_{\mathbb{R}^{2 N}} \frac{C_{N, s, p}|u(x)-u(y)|^{p-2}(u(x)-u(y))(v(x)-v(y))}{|x-y|^{N+s p}} \d x \d y\right) \d \mu^{+}(s) \\
& -\int_{[0, \bar{s}]}\left(\iint_{\mathbb{R}^{2 N}} \frac{C_{N, s, p}|u(x)-u(y)|^{p-2}(u(x)-u(y))(v(x)-v(y))}{|x-y|^{N+s p}} \d x \d y\right) \d \mu^{-}(s) \\
& =\int_\Omega f(x,u)v \d x, \text{ for all } v \in X_{p}(\Omega).
\end{align*}
\end{definition}
The weak solutions to the problem \eqref{p} are the critical points of the energy functional $\mathcal{I}: X_{p}(\Omega) \rightarrow \mathbb{R}$, which is defined as
\begin{equation}\label{functional MP}
\mathcal{I}(u):=\frac{1}{p}\left[\rho_{p}(u)\right] ^{p}-\frac{1}{p} \int_{[0, \bar{s}]}[u]_{s, p}^{p} \d \mu ^{-}(s)-\int_\Omega\int_0^uf(x,\tau)\d \tau \d x.
\end{equation}
Furthermore, we have
\begin{align*}
\langle \mathcal{I}^{\prime}(u), v\rangle&= \int_{[0,1]}\left(\iint_{\mathbb{R}^{2 N}}  \frac{C_{N, s, p}|u(x)-u(y)|^{p-2}(u(x)-u(y))(v(x)-v(y))}{|x-y|^{N+s p}} \d x \d y\right) \d \mu^{+}(s) \\
& -\int_{[0, \bar{s}]}\left(\iint_{\mathbb{R}^{2 N}} \frac{C_{N, s, p}|u(x)-u(y)|^{p-2}(u(x)-u(y))(v(x)-v(y))}{|x-y|^{N+s p}} \d x \d y\right) \d \mu^{-}(s) \\
& -\int_{\Omega}f(x,u) v \d x, \text{ for all } u, v \in X_{p}(\Omega).
\end{align*}
\begin{rem}
    We mention here that first term of the integral on the LHS in Definition \ref{weaksol MP} has the following explicit form: 
\begin{align*} 
&\int_{[0,1]}\left(\iint_{\mathbb{R}^{2 N}} \frac{C_{N, s, p}|u(x)-u(y)|^{p-2}(u(x)-u(y))(v(x)-v(y))}{|x-y|^{N+s p}} \d x \d y\right) \d \mu^{+}(s)\\
&=\int_{(0,1)}\left(\iint_{\mathbb{R}^{2 N}} \frac{C_{N, s, p}|u(x)-u(y)|^{p-2}(u(x)-u(y))(v(x)-v(y))}{|x-y|^{N+s p}} \d x \d y\right) \d \mu^{+}(s)\\ 
&+\mu^+(\{0\}) \int_\Omega {|u(x)|^{p-2}}u(x)v(x) \d x + \mu^+(\{1\}) \int_\Omega |\nabla u(x)|^{p-2}\nabla u(x) \cdot \nabla v(x) \d x. \end{align*} 
Similarly, the second term of the integral on the LHS is defined but for simplicity, we will continue to misuse the expression as mentioned in Definition \ref{weaksol MP}.
\end{rem}

Finally, we assume that the nonlinearity $f:\bar{\Omega}\times\mathbb{R}\rightarrow\mathbb{R}$ satisfies the following conditions:
\begin{align}
 \text{(i)}&~~f\in C(\bar{\Omega}\times\mathbb{R})\label{f1},\\
\text{(ii)}&~~|f(x,t)|\leq a_1+a_2|t|^{q-1}, \text{ for some } a_1,a_2>0  \text{ with } p<q<p^*_{s_\sharp}:=\frac{Np}{N-ps_\sharp}\label{f2},\\
  \text{(iii)}&~~ f(x,-t)=-f(x,t),  \text{ for all } x\in\Omega, t\in\mathbb{R}\label{f3}.
\end{align}
Condition (ii) reflects the assumption of subcritical growth. On the other hand, Condition (iii) is essential for establishing the existence of infinitely many solutions, as it requires a certain symmetry of $f$. It is also immediate to observe that this condition implies $f(x,0)=0$, which ensures that $u=0$ is always a trivial solution of \eqref{p}. Our objective, therefore, is to establish the existence of nontrivial solutions to \eqref{p}.

Additionally, we assume that $f$ verifies the famous {\it Ambrosetti-Rabinowitz (AR)} condition; that is, there exist constants $\mu > p$ and $r > 0$ such that, for all $x \in \Omega$ and $t \in \mathbb{R}$ with $|t| \geq r$, we have

\begin{equation}\label{f4}
    0<\mu F(x,t)\leq tf(x,t), 
\end{equation}
  where $F$ is the primitive of $f$ with respect to the second variable, that means, 
  $$F(x, t)= \int_0^t f(x, \tau) \,d\tau.$$ 
  The Ambrosetti–Rabinowitz (AR) condition, introduced in the seminal work of Ambrosetti and Rabinowitz \cite{AR73}, plays a central role in guaranteeing the boundedness of Palais–Smale sequences associated with the energy functional of the problem under consideration. In that work, by applying the celebrated Mountain Pass Theorem, the authors proved the existence of nontrivial solutions to problem \eqref{p} in the case $A_{\mu, p} = -\Delta$, assuming that the nonlinearity satisfies superlinear growth while remaining subcritical.

  This result was later extended to the nonlocal framework of the fractional Laplacian, $A_{\mu, p} = (-\Delta)^s$, by Servadei and Vandinoci \cite{SV12}. Subsequently, Binlin et al. \cite{BMS2015} investigated the existence of infinitely many solutions to problem \eqref{p} for the case $A_{\mu, p} = (-\Delta)^s$. Furthermore, in \cite{MMTZ16}, the authors established the existence of a sequence of nontrivial solutions for Kirchhoff-type problems driven by the fractional $p$-Laplacian, which in particular includes problem \eqref{p} with $A_{\mu, p} = (-\Delta)_p^s$. In this paper, our goal is to generalize and unify the aforementioned results by studying problem \eqref{p} under the assumption that $f$ satisfies subcritical and superlinear growth conditions together with the Ambrosetti–Rabinowitz (AR) condition. As an example can take  $f(x, t):=a(x) |t|^{q-2}t$ with $a \in C(\overline{\Omega})$ and $q \in (p, p^*_{s_\sharp}).$ Our framework is sufficiently flexible to encompass both local and nonlocal operators, thereby extending classical variational results to a wider class of mixed and fractional problems. Now, we state one of our main results as follows. 
\begin{thm}\label{t1} Let $\Omega \subset \mathbb{R}^N$ be a bounded open subset. 
Let $\mu = \mu^{+} - \mu^{-}$, where $\mu^{+}$ and $\mu^{-}$ satisfy conditions \eqref{mu 1}, \eqref{mu 2} and \eqref{mu 3}.  
Assume that $s_{\sharp}$ is given as in \eqref{mu 4}, and suppose that $1<p < \frac{N}{s_\sharp}$ and $1<p < q < p^*_{s_\sharp} := \frac{Np}{N - ps_\sharp}$.  
Suppose further that $f$ satisfies \eqref{f1}, \eqref{f2}, \eqref{f3} and \eqref{f4}.  
Then, there exists $\kappa_* \geq 0$ such that, for all $\kappa \in [0, \kappa_*]$, the problem \eqref{p} admits an infinite sequence of large-energy solutions in $X_p(\Omega).$
\end{thm}

We recall that by an infinite sequence of large-energy ( or high-energy) solutions we mean a sequence $(u_j)_{j \in \mathbb{N}} \subset X_p(\Omega)$ of weak solutions to problem~\eqref{p} such that $\mathcal{I}(u_j) \to +\infty$ as $j \to \infty$. Moreover, under the symmetry assumption \eqref{f3}, if $u$ is a weak solution, then so is $-u$. Consequently, this result (and other results of this paper) guarantee the existence of infinitely many symmetric pairs of high-energy weak solutions $\{u_j, -u_j\}_{j \in \mathbb{N}}$.

To put things in context, let us briefly describe the literature for existence of solutions to elliptic problem involving $A_{\mu, p}.$ Consider the following critical  problem  
\begin{align}\label{problem 2}
A_{\mu,p} u &= \lambda |u|^{q-2}u+|u|^{p^*_{s_\sharp}-2}u ~\text{in}~\Omega  \nonumber\\
 u & =0 ~\text{in}~ \mathbb{R}^N \setminus \Omega,
 \end{align}
where $\Omega$ is a bounded open subset of $\mathbb{R}^{N}$, $0<\lambda$, $0\leq s\leq 1<p<\frac{N}{s_\sharp}.$ Dipierro \textit{et al.} \cite{DPSV} investigated the linear case of \eqref{problem 2}, corresponding to $q=p$, and proved the existence of $\kappa_0>0$ such that, for every $\kappa \in [0,\kappa_0]$, problem \eqref{problem 2} admits $m$ pairs of distinct nontrivial solutions. Subsequently, the last two authors, together with Aikyn and Ruzhansky \cite{AGKR2025}, studied the superlinear case $p<q<p^*_{s_\sharp}$ and established the existence of a positive constant $\lambda^*$ with the property that, for any $\lambda \geq \lambda^*$, there exists a sufficiently small $\kappa_0>0$ such that, for all $\kappa \in [0,\kappa_0]$, problem \eqref{problem 2} possesses at least one nontrivial solution. More recently, in \cite{BGK2025}, we analyzed the sublinear case $1<q<p<p^*_{s_\sharp}$ and demonstrated that problem \eqref{problem 2} admits infinitely many nontrivial solutions in $X_p(\Omega)$ with negative energy. In addition, we mention \cite{ABB2025}, where the authors considered the linear case $p=2$ and established existence and multiplicity results by replacing the critical exponent term with a general nonlinearity satisfying subcritical growth and asymptotic linearity at infinity. In this context, Theorem~\ref{t1} may be regarded as a natural continuation of this line of research.

The following easy consequence of Theorem \ref{t1} for the mixed local and nonlocal operator $( -\Delta+(-\Delta)^s)$ and mixed local and nonlocal $p$-Laplace operator $( -\Delta_{p} +(-\Delta_{p})^s)$ also seems to be new in the literature.  In this situation, we choose $\mu:=\delta_1+\delta_s,$ where $\delta_1$ and $\delta_s$ denote the Dirac measures centred at $1$ and $s \in (0, 1),$ respectively.  Consequently, $\mu$ satisfying all of the conditions \eqref{mu 1}-\eqref{mu 3} with $\bar{s}:=1$ and $\kappa:=0.$  We also take $s_\sharp:=1.$ The following corollary complements the findings of \cite{MMV23}, where the existence of a weak solution was established for problems driven by mixed local and nonlocal operators.

\begin{cor}
    Let $\Omega \subset \mathbb{R}^N$ be a bounded open subset.   
Suppose that $1<p < N$ and $1<p < q < p^* := \frac{Np}{N - p}$.  
Suppose further that $f$ satisfies \eqref{f1}, \eqref{f2}, \eqref{f3} and \eqref{f4}.  
Then,  the problem \begin{align}\label{pmixed}
( -\Delta_{p} +(-\Delta_{p})^s) u &= f(x,u) ~\text{in}~\Omega  \nonumber\\
 u & =0 ~\text{in}~ \mathbb{R}^N \setminus \Omega,
 \end{align} admits an infinite sequence of large-energy solutions.
\end{cor}

Although the Ambrosetti--Rabinowitz (AR) condition is widely regarded as natural, it is nevertheless restrictive, as it excludes a broad range of nonlinearities. From the AR condition \eqref{f4}, one can deduce the existence of constants $a_3, a_4 > 0$ such that  
\begin{equation}\label{f5}
    F(x,t) \geq a_3 |t|^\mu - a_4, \quad \text{for all } x \in \bar{\Omega}, \; t \in \mathbb{R}.
\end{equation}
However, a number of nonlinearities fail to satisfy the AR condition while still exhibiting superlinear growth at infinity. Indeed, combining \eqref{f5} with the assumption $\mu > p$ yields  
\begin{equation}\label{f6}
    \lim_{|t| \to \infty} \frac{F(x,t)}{|t|^p} = \infty, \quad \text{uniformly for } x \in \bar{\Omega}.
\end{equation}
 It is evident that condition \eqref{f6} is weaker than \eqref{f5}. For example, the function  
\begin{equation} \label{exnon}
    f(x,t) = |t|^{p-2} t \, \log\!\left(1 + |t|^{p-1}\right)
\end{equation}
satisfies the superlinear growth condition \eqref{f6}, yet it does not fulfill \eqref{f5}. This illustrates that nonlinearities governed by \eqref{f6} extend beyond those admissible under the Ambrosetti--Rabinowitz condition \eqref{f4}.  

This consideration provides a strong motivation to investigate the existence of solutions under weaker superlinearity assumptions, thereby enlarging the class of admissible nonlinearities while preserving the variational framework necessary for the analysis.  To address the superlinear problem \eqref{p} in this more general framework, we impose the following assumption on $f$ introduced by Jeanjean \cite{J99}:

There exists $\alpha\geq 1$ such that for any $x \in\Omega$, we have
\begin{equation}\label{f7}
    \mathcal{F}(x,t')\leq \alpha\mathcal{F}(x,t)~~\text{for any } t,t'\in\mathbb{R} ~~\text{with } 0<t'\leq t,
\end{equation}
where $$\mathcal{F}(x,t)=\frac{1}{p}tf(x,t)-F(x,t).$$ It is worth mentioning that problems whose nonlinearities satisfy condition \eqref{f7} are commonly referred to in the literature as {\it problems without the AR condition}. We note the function \eqref{exnon} also satisfies the condition \eqref{f7}. There are several results in the literature dealing with nonlinear nonlocal elliptic problems without the AR condition, we refer to \cite{BMS2015, FL09,J99,L2010,MS08} and references therein.

Our second result is about the existence of the large-energy solution to the problem \eqref{p} without the AR condition.
\begin{thm}\label{t2} 
Let $\Omega$ be a bounded open subset of $\mathbb{R}^{N}$.  
Let $\mu = \mu^{+} - \mu^{-}$, where $\mu^{+}$ and $\mu^{-}$ satisfy conditions \eqref{mu 1}--\eqref{mu 3}.  
Assume that $s_{\sharp}$ is given as in \eqref{mu 4}, and that $1<p < \tfrac{N}{s_\sharp}$ and $1<p < q < p^*_{s_\sharp}:= \frac{Np}{N - ps_\sharp}$.  
Suppose that $f$ satisfies \eqref{f1}, \eqref{f2}, \eqref{f3}, \eqref{f6}, and \eqref{f7}.  
Then there exists $\kappa_* \geq 0$ such that, for every $\kappa \in [0,\kappa_*]$, the problem \eqref{p} admits an infinite sequence of large-energy solutions in $X_p(\Omega)$.

\end{thm}
Observe that the following condition is stronger than condition \eqref{f6} but weaker than AR condition \eqref{f4},
\begin{equation}\label{f6'}
    \lim\limits_{|t|\rightarrow +\infty}\frac{f(x,t)}{|t|^{p-2}t}=+\infty,\text{   uniformly in  } x\in\bar{\Omega}.
    \end{equation}
    This is commonly referred to as the \emph{superlinear growth condition} of $f$ at infinity, particularly in the case $p=2$. Thus, we have the following result. 
\begin{cor}  Let $\Omega$ be a bounded open subset of $\mathbb{R}^{N}$.  
Consider the measure $\mu = \mu^{+} - \mu^{-}$, where $\mu^{+}$ and $\mu^{-}$ satisfy \eqref{mu 1}--\eqref{mu 3}.  
Let $s_{\sharp}$ be given by \eqref{mu 4}, and suppose that $1<p < \tfrac{N}{s_\sharp}$ and $1<p < q < p^*_{s_\sharp}$.  
Assume further that $f$ satisfties conditions \eqref{f1}, \eqref{f2}, \eqref{f3}, \eqref{f7}, and \eqref{f6'}.  
Then there exists a constant $\kappa_* \geq 0$ such that, for all $\kappa \in [0,\kappa_*]$, the problem \eqref{p} has an infinite sequence of large-energy solutions in $X_p(\Omega)$.

\end{cor}

Motivated by \cite{LL2003}, we now introduce the following global condition on $f$: the function 
\begin{equation}\label{f8}
    g(t)=\frac{f(x,t)}{|t|^{p-2}t}~~\text{is increasing in }t\geq 0 ~~\text{and is decreasing in }t\leq0.  
\end{equation}
Observe that condition \eqref{f7} is weaker than \eqref{f8}. Since, we are dealing with the problem \eqref{p} in a bounded domain, we replace the global condition \eqref{f8} by the following modified condition: there exists $t''>0$ such that for any $x \in\Omega$, the function 
\begin{equation}\label{f9}
    g(t)=\frac{f(x,t)}{|t|^{p-2}t}~~\text{is increasing in }t\geq t'' ~~\text{and is decreasing in }t\leq -t''.  
\end{equation}
The condition \eqref{f9} was introduced in \cite{L2010}.
The following theorem is our last result of this paper.
\begin{thm}\label{t3} 
Let $\Omega$ be a bounded open subset of $\mathbb{R}^{N}$.  
Let $\mu = \mu^{+} - \mu^{-}$, with $\mu^{+}$ and $\mu^{-}$ satisfying conditions \eqref{mu 1}--\eqref{mu 3}.  
Assume that $s_{\sharp}$ is given as in \eqref{mu 4}, and that $1<p < \frac{N}{s_\sharp}$ and $1<p < q < p^*_{s_\sharp}$.  
Suppose $f$ satisfies \eqref{f1}, \eqref{f2}, \eqref{f3}, \eqref{f6}, and \eqref{f9}.  
Then, there exists $\kappa_* \geq 0$ such that, for all $\kappa \in [0, \kappa_*]$, the problem \eqref{p} possesses an infinite sequence of large-energy solutions in $X_p(\Omega)$.
\end{thm}

The strategy employed to establish results above (Theorems \ref{t1}, \ref{t2}, and \ref{t3}) is based on the search for infinitely many critical points of the Euler–Lagrange functional \eqref{functional MP} associated with problem \eqref{p}. For this purpose, we make use of the Fountain Theorem due to Bartsch \cite{B1993} discussed in Section \ref{s4}. As is customary in the application of critical point theorems, the analysis requires a careful investigation of both the compactness properties and the geometric features of the functional. Concerning compactness, we show that the functional satisfies the classical Palais–Smale  condition when the nonlinearity verifies the Ambrosetti–Rabinowitz assumption \eqref{f4}. Recall that  the functional $\mathcal{I}$ given by \eqref{functional MP} satisfies the Palais--Smale condition $\mathrm{(PC)_c}$ at the level $c\in\mathbb{R}$ if every sequence $\{u_j\}_{j\in\mathbb{N}}\subset X_p(\Omega)$ such that
\begin{equation}\label{PS-seq}
\mathcal{I}(u_j)\to c
\quad\text{and}\quad
\mathcal{I}'(u_j) \to 0 \,\, \text{in}\,\,X_p(\Omega)^*
\quad\text{as } j\to+\infty,
\end{equation}
admits a subsequence which converges strongly in $X_p(\Omega)$.

On the other hand, when the right-hand side satisfies alternative superlinear growth conditions (that is, \eqref{f6}, \eqref{f7}, \eqref{f6'}, \eqref{f9}), compactness is ensured through the Cerami condition.  The functional $\mathcal{I}$ satisfies the Cerami condition $\mathrm{(Ce)_c}$ at the level $c\in\mathbb{R}$ if every sequence $\{u_j\}_{j\in\mathbb{N}}\subset X_p(\Omega)$ such that
\begin{equation}\label{C-seq}
\mathcal{I}(u_j)\to c
\quad\text{and}\quad
\big(1+\|u_j\|_{X_p(\Omega)}\big)
 \mathcal{I}'(u_j)\to 0\,\, \text{in}\,\,X_p(\Omega)^*
\quad\text{as } j\to+\infty,
\end{equation}
admits a subsequence which converges strongly in $X_p(\Omega)$.
The Cerami \cite{C1978} compactness condition was introduced by Cerami  as a weaker alternative to the Palais--Smale condition. The Fountain Theorem remains applicable under the Cerami hypothesis. 
In both settings, the main technical difficulty lies in proving the boundedness of the Palais–Smale or Cerami sequences and in handling the superposition operator due to its general nature, a step that is crucial for applying variational methods effectively.


\subsection*{Applications and Examples}
Now, we apply the abstract results established in the previous sections to a number of illustrative examples, thereby deriving new existence results that depend on the specific choice of the measure $\mu$. As a first step, we present particular instances of Theorem~\ref{t3}. Analogous applications of Theorems~\ref{t1} and \ref{t2} can be discussed in the same framework. 
For convenience, we shall employ the unified notation $\mathbb{X}(\Omega)$ to denote the (fractional) Sobolev space associated with each example under consideration. It is important to note that the precise definition of $\mathbb{X}(\Omega)$ depends on the operator involved, and must be specified in terms of the underlying space $X_p(\Omega)$ in each case.

\begin{cor}
    Let $\Omega \subset \mathbb{R}^N$ be a bounded open subset.   
Suppose that $1<p < N$ and $1<p < q < p^* := \frac{Np}{N - p}$.  
Suppose $f$ satisfies \eqref{f1}, \eqref{f2}, \eqref{f3}, \eqref{f6}, and \eqref{f9}. 
Then,  the problem \begin{align}\label{pmixed1}
( -\Delta_{p} +(-\Delta_{p})^s) u &= f(x,u) ~\text{in}~\Omega  \nonumber\\
 u & =0 ~\text{in}~ \mathbb{R}^N \setminus \Omega,
 \end{align} admits an infinite sequence of large-energy solutions in $\mathbb{X}(\Omega)$.
\end{cor}

We show that we can also consider nonlocal operators associated with a convergent series of Dirac measures.
For applying Theorem \ref{t3}, we define 
 $$\mu:= \sum_{k=0}^{+\infty} a_k \delta_{s_k},$$
 where $\delta_{s_k}$ represents the Dirac measures at $s_k.$   Subsequently, we observe that $\mu$ fulfills all the conditions \eqref{mu 1}-\eqref{mu 3} by choosing 
 $\bar{s}:=s_0$, $\kappa:=0$ and $s_\sharp := s_0$. Thus, we get the following result.
\begin{cor}  Let $ \Omega$ be a bounded subset of $\mathbb{R}^N$ and $1 \geq s_0>s_1>s_2> \ldots \geq 0.$ Assume that the operator 
    $$\sum_{k=0}^{+\infty} a_k (-\Delta_p)^{s_k}\quad \text{with} \quad \sum_{k=0}^{+\infty} a_k \in (0, +\infty),$$
    where $a_0>0$ and $a_k \geq 0$ for all $k \geq 1$ with $p_{s_0}^*:= \frac{pN}{N-ps_{0}}$ be the fractional critical Sobolev exponent. Let $1<p<\frac{N}{s_0}$ and $1<p<q <p^*_{s_0}.$  Suppose $f$ satisfies \eqref{f1}, \eqref{f2}, \eqref{f3}, \eqref{f6}, and \eqref{f9}.  Then the problem \begin{align}\label{problem52}
\sum_{k=0}^{+\infty} a_k (-\Delta_p)^{s_k}  u &= f(x, u) ~\text{in}~\Omega  \nonumber\\
 u & =0 ~\text{in}~ \mathbb{R}^N \setminus \Omega,
 \end{align}  possesses an infinite sequence of large-energy solutions in $\mathbb{X}(\Omega)$.
\end{cor}

An attractive scenario arises when the measure $\mu$ changes sign.  This means, for example, that the operator could use a minor term with a ``wrong" sign. To our knowledge, no existing literature addresses a result of this nature, even for the situation $p=2.$ For this,  we take $\mu:=\delta_1-\alpha \delta_s,$ where $\delta_1$ and $\delta_s$ are two Dirac measures centered at $1$ and $s \in [0, 1),$ respectively. Now, we choose $\bar{s}:=1$ and $s_\sharp:=1,$ so that the conditions  \eqref{mu 1} and \eqref{mu 2} hold. Moreover, note that 
    $$\mu^-([0, \bar{s}]) \leq \max \{0, \alpha\} = \max\{0, \alpha\} \mu^+([\bar{s}, 1]),$$
    which says that condition \eqref{mu 3} holds by choosing $\kappa:= \max\{0, \alpha\}.$ Hence, the following result follows as a direct application of Theorem \ref{t3}.
\begin{cor}
Let $\Omega$ be a bounded open subset of $\mathbb{R}^{N}$.  
Let $\mu = \mu^{+} - \mu^{-}$, with $\mu^{+}$ and $\mu^{-}$ satisfying conditions \eqref{mu 1}--\eqref{mu 3}.  
Assume that  $1<p < N$ and $1<p < q < p^*:=\frac{Np}{N-p}$.  
Suppose $f$ satisfies \eqref{f1}, \eqref{f2}, \eqref{f3}, \eqref{f6}, and \eqref{f9}. Then, for all $\alpha \geq 0,$ there exists  $\kappa_* \geq 0$ such that  for all $\kappa \in [0, \kappa_*]$, the problem \begin{align}\label{problem54}
 -\Delta_{p}u -\alpha (-\Delta_{p})^s u  &= f(x, u) ~\text{in}~\Omega  \nonumber\\
 u & =0 ~\text{in}~ \mathbb{R}^N \setminus \Omega,
 \end{align} possesses an infinite sequence of large-energy solutions in $\mathbb{X}(\Omega)$.
\end{cor}

We highlight another intriguing result arising from the continuous superposition of fractional operators of the $p$-Laplacian type. To the best of our knowledge, this result is also novel.  In this case, we define $d\mu(s):=g(s) ds$ with $g$ as in the statement of the following corollary. Thus, the operator $A_{p, \mu}$ becomes
    $$ \int_0^1 g(s) (-\Delta_p)^s u\, ds.$$
    Thanks to the conditions presented in \eqref{condf}, all the conditions \eqref{mu 1}-\eqref{mu 3} are satisfied by taking $\bar{s}:=s_\sharp.$  Thus, the following result is concluded from Theorem \ref{t3}.

\begin{cor}
 Let $ \Omega$ be a bounded subset of $\mathbb{R}^N.$    Let $s_\sharp \in (0, 1), \kappa \geq 0,$  and let $g \not\equiv 0 $ be a measurable function such that 
    \begin{align} \label{condf}
        &g \geq 0 \quad \text{in} \quad (s_\sharp, 1), \nonumber \\
        &\int_{s_\sharp}^1 g(s)\, ds >0, \\
        \text{and}\quad &\int_0^{s_\sharp} \max\{0, -g(s)\}\, ds \leq \kappa \int_{s_\sharp}^1 g(s)\, ds. \nonumber
        \end{align}
    Suppose that $1<p<\frac{N}{s_{\sharp}}$ and $1<p<q < p_{s_\sharp}^*$, where $p_{s_\sharp}^*=\frac{p N}{N-{s_\sharp}p}$ is the fractional critical Sobolev exponent.
Then,  there exists  $\kappa_* \geq 0$ such that  for all $\kappa \in [0, \kappa_*]$, the problem
\begin{equation}\label{p2.1}
 \begin{cases}
     \bigintss_0^1 g(s) (-\Delta_p)^s u\, ds  = f(x,u) \quad & \text { in } \Omega,  \\ \quad
    u  = 0  &\text { in } \mathbb{R}^{N},\backslash \Omega 
 \end{cases}
\end{equation}
admits an infinite sequence of large-energy solutions in $\mathbb{X}(\Omega)$.
    
\end{cor}

 We conclude this introduction by outlining the organization of the paper. In Section \ref{pre}, we introduce the functional framework of the problem, recall several preliminary results, and establish a new norm equivalent to the standard one. Section \ref{sec3} is devoted to analyzing the compactness properties of the associated energy functional. In Section \ref{s4}, we present the Fountain Theorem and derive a key auxiliary result. Finally, the last section contains the proofs of the main results.

\section{Functional Analytic setting for the superposition operator} \label{pre}

This section is devoted to the construction of the appropriate fractional Sobolev spaces and the discussion of their fundamental properties, which are essential for the analysis of our problem. Similar treatments can be found in \cite{AGKR2025, DPSV, DPSV2, DPSV1}. To begin, for $s \in [0,1]$, we define
\begin{equation*}
    [u]_{s, p}:= \begin{cases}\|u\|_{L^{p}\left(\mathbb{R}^{N}\right)} & \text { if } s=0, \\ \left(C_{N, s, p} \iint_{\mathbb{R}^{2 N}} \frac{|u(x)-u(y)|^{p}}{|x-y|^{N+s p}} \d x \d y\right)^{1 / p} & \text { if } s \in(0,1), \\ \|\nabla u\|_{L^{p}\left(\mathbb{R}^{N}\right)} & \text { if } s=1,\end{cases}
\end{equation*}
where $C_{N,s,p}$ is the normalizing constant and we have
$$\lim _{s \searrow 0^+}[u]_{s, p}=[u]_{0, p} \quad \text { and } \quad \lim _{s \nearrow 1^-}[u]_{s, p}=[u]_{1, p}.$$

We define the space $X_{p}(\Omega)$ as the collection of measurable functions $u: \mathbb{R}^{N} \rightarrow \mathbb{R}$ such that $u=0$ in $\mathbb{R}^{N} \backslash \Omega$, for which the following norm is finite:
\begin{equation}\label{norm on Xp} \|u\|_{X_p(\Omega)} = \rho_{p}(u):=\left(\int_{[0,1]}[u]_{s, p}^{p} \d \mu^{+}(s)\right)^{1 / p}< +\infty.
\end{equation}
The space $X_p(\Omega)$ was initially introduced in \cite{DPSV2} for $p=2$ and subsequently generalized in \cite{DPSV} for $1<p<\infty$. We present the following results from \cite{DPSV} (see also \cite{AGKR2025}), which are necessary for our study.
\begin{lem}\cite{AGKR2025, DPSV}\label{Uniform convexity}
    The $X_{p}(\Omega)$ is a separable Banach space for $1 \leq p < \infty$ and is a uniformly convex Banach space, for $1<p < \infty$.
\end{lem}

\begin{lem}\label{reabsorb} \cite[Proposition 4.1]{DPSV}
    Let $1<p<N$ and \eqref{mu 2} and \eqref{mu 3} be true. Then, there exists $c_{0}=c_{0}(N, \Omega, p)>0$ such that, for any $u \in X_{p}(\Omega)$, we have
\begin{equation*}
    \int_{[0, \bar{s}]}[u]_{s, p}^{p} \d \mu^{-}(s) \leq c_{0} \kappa \int_{[\bar{s}, 1]}[u]_{s, p}^{p} \d \mu(s)=c_{0} \kappa \int_{[\bar{s}, 1]}[u]_{s, p}^{p} \d \mu^+(s).
\end{equation*}
\end{lem}
\begin{lem}\label{Sobolev emb} \cite[Theorem 3.2]{DPSV}
Let $\Omega$ be an open, bounded subset of $\mathbb{R}^{N}$ and $p \in(1, N)$.
Then, there exists $C=C(N, \Omega, p)>0$ such that, for every $s_1, s_2 \in[0,1]$ with $s_1 \leq s_2$ and every measurable function $u: \mathbb{R}^{N} \rightarrow \mathbb{R}$ with $u=0$ a.e. in $\mathbb{R}^{N} \setminus \Omega$, we have
$$
[u]_{s_1, p} \leq C[u]_{s_2, p}.
$$
\end{lem} 
\begin{prop}\cite[Proposition 2.5]{AGKR2025}\label{compact and cont embedding}
Assume that \eqref{mu 1}–\eqref{mu 4} hold. Then, there exists a positive constant $\bar{c}=\bar{c}\left(N, \Omega, s_{\sharp}, p\right)$ such that, for any $u \in X_{p}(\Omega)$, we have
\begin{equation} \label{Xpusialemb}
[u]_{s_{\sharp},p} \leq \bar{c}\left(\int_{[0,1]}[u]_{s,p}^{p} \mathrm{~d} \mu^{+}(s)\right)^{\frac{1}{p}}.
\end{equation}
 In particular, the embedding 
\begin{equation}\label{embedding}
    X_p(\Omega) \hookrightarrow L^{r}(\Omega)
\end{equation}
is continuous for all $r \in[1,p_{s_{\sharp}}^{*}]$ and compact for all $r \in[1,p_{s_{\sharp}}^{*})$.
\end{prop}

We now recall the following weak convergence results from \cite{DPSV}. 
\begin{lem} \label{weak convergence}  \cite[ Lemma 5.8]{DPSV}
    Let $(u_{n})$ be a bounded sequence in $X_{p}(\Omega)$. Then, there exists $u \in X_{p}(\Omega)$ such that for all $v \in X_{p}(\Omega)$, we have
\begin{align}
    \lim _{n \rightarrow\infty} &\int_{[0,1]}\left(\iint_{\mathbb{R}^{2 N}} \frac{\left|u_{n}(x)-u_{n}(y)\right|^{p-2}\left(u_{n}(x)-u_{n}(y)\right)(v(x)-v(y))}{|x-y|^{N+s p}} \d x \d y\right) \d \mu^{ \pm}(s)\nonumber \\
&=\int_{[0,1]}\left(\iint_{\mathbb{R}^{2 N}} \frac{|u(x)-u(y)|^{p-2}(u(x)-u(y))(v(x)-v(y))}{|x-y|^{N+s p}} \d x \d y\right) \d \mu^{ \pm}(s).
\end{align}
\end{lem}
We state the following Brezis-Lieb type lemma from \cite{DPSV}. 
\begin{lem}\label{B-L lemma} \cite[ Lemma 5.9]{DPSV}
   Let $u_{n}$ be a bounded sequence in $X_{p}(\Omega)$. Suppose that $u_{n}$ converges to some $u$ a.e. in $\mathbb{R}^{N}$ as $n \rightarrow\infty$. Then we have
\begin{equation}
\int_{[0,1]}[u]_{s, p}^{p} \d \mu^{ \pm}(s)=\lim _{n \rightarrow\infty}\left(\int_{[0,1]}\left[u_{n}\right]_{s, p}^{p} \d \mu^{ \pm}(s)-\int_{[0,1]}\left[u_{n}-u\right]_{s, p}^{p} \d \mu^{ \pm}(s)\right).
\end{equation}
\end{lem}

In \cite{BGK2025}, we establish a novel norm on $X_p(\Omega)$ that is equivalent to the norm $\rho_p$ as defined in \eqref{norm on Xp}.
    \begin{lem}\label{lmn2.8}\cite[Lemma 2.7]{BGK2025}
        Let $\mu$ satisfies the assumptions \eqref{mu 1}-\eqref{mu 4}. Then there exists a sufficiently small $\kappa_0(N,\Omega,p)$ such that for all $\kappa\in [0,\kappa_0]$, we have a norm $\eta_p$ on $X_p(\Omega)$, which is equivalent to the norm $\rho_p$ as in \eqref{norm on Xp}, where
      \begin{align}\label{eq2.6}
         \eta_{p}(u):=\left[\mathcal{H}(u)\right]^\frac{1}{p}=\left(\int_{[0,1]}[u]_{s, p}^{p} \d \mu^{+}(s)-\int_{[0,\bar{s})}[u]_{s, p}^{p} \d \mu^{-}(s)\right)^{1 / p}.
    \end{align}
    \end{lem}
\begin{rem}
    Note that for $\kappa=0$, we have $[\rho_p(u)]=[\eta_p(u)]$ for any $u\in X_p(\Omega)$. In case of $p=2$, we  define $[\eta_2(u)]:=\sqrt{\langle u,u \rangle}$ for any $u\in X_2(\Omega)$, where $${\langle u,v\rangle}:={\langle u,v\rangle}^+-{\langle u,v\rangle}^-$$ 
    with 
    $${\langle u,v\rangle}^+:=\int_{[0,1]}\left(\iint_{\mathbb{R}^{2 N}} \frac{C_{N, s}(u(x)-u(y))(v(x)-v(y))}{|x-y|^{N+2s}} \d x \d y\right) \d \mu^{+}(s)$$
    $${\langle u,v\rangle}^-:=\int_{[0, \bar{s}]}\left(\iint_{\mathbb{R}^{2 N}} \frac{C_{N, s}(u(x)-u(y))(v(x)-v(y))}{|x-y|^{N+2s}} \d x \d y\right) \d \mu^{-}(s).$$
    Therefore, $X_2(\Omega)$ reduces to a Hilbert space concerning the inner product $\langle u,v \rangle$ for any $\kappa \in [0,\kappa_0]$, where $\kappa_0$ is sufficiently small. Further properties of the Hilbert space $X_2(\Omega)$ and related problems, we refer to \cite{ABB2025, AGKR2025,  DPSV2025}. 
\end{rem}

\section{Important Results: Compactness conditions} \label{sec3}
This section is devoted to verifying the compactness properties of the energy functional $\mathcal{I}$ associated with problem~\eqref{p}, namely the Palais-Smale and Cerami conditions, which are essential for the application of the Fountain Theorem.

We begin by considering the case in which the nonlinearity $f$ satisfies the Ambrosetti-Rabinowitz condition \eqref{f4}. In this setting, we demonstrate that the energy functional $\mathcal{I}$ satisfies the Palais--Smale $\mathrm{(PS)}$ condition.
\begin{lem}\label{lmn3.1}
   Suppose that $f:\bar{\Omega}\times \mathbb{R}\rightarrow \mathbb{R}$ satisfies the conditions \eqref{f1}, \eqref{f2}, and \eqref{f4}. Then, the energy functional $\mathcal{I}$ verifies the Palais-Smale $\mathrm{(PS)}$ condition at any level $c\in \mathbb{R}$.
\end{lem}
\begin{proof}
    Let $(u_{n}) \subset X_{p}(\Omega)$ be a $\mathrm{PS}$ sequence at $c\in\mathbb{R}$ for the energy functional $\mathcal{I}$. Thus, we get 
\begin{align}\label{level c}
    \lim _{n \rightarrow\infty} \mathcal{I}\left(u_{n}\right) =&\lim _{n \rightarrow\infty} \Bigg[\frac{1}{p}\left[\rho_{p}\left(u_{n}\right)\right]^{p}-\frac{1}{p} \int_{[0, \bar{s}]}\left[u_{n}\right]_{s, p}^{p} \d \mu^{-}(s) -\int_\Omega\int_0^{u_n}f(x,\tau)\d \tau \d x \Bigg]=c
\end{align}
and
\begin{align}\label{derivative}
 \lim _{n \rightarrow\infty}& \sup_{v \in X_p(\Omega)} \langle \mathcal{I}^{\prime}\left(u_{n}\right), v\rangle \nonumber \\
=&  \lim _{n \rightarrow\infty} \sup_{v \in X_p(\Omega)} \Bigg[ \int_{[0,1]}\Bigg( \iint_{\mathbb{R}^{2 N}}  \frac{\splitfrac{C_{N, s, p}\left|u_{n}(x)-u_{n}(y)\right|^{p-2}\left(u_{n}(x)-u_{n}(y)\right)}{\times(v(x)-v(y))}}{|x-y|^{N+s p}}\d x \d y\Bigg) \d \mu^{+}(s) \nonumber\\ 
& -\int_{[0, \bar{s}]}\Bigg( \iint_{\mathbb{R}^{2 N}} \frac{\splitfrac{C_{N, s, p}\left|u_{n}(x)-u_{n}(y)\right|^{p-2}\left(u_{n}(x)-u_{n}(y)\right)}{\times(v(x)-v(y))}}{|x-y|^{N+s p}} \d x \d y\Bigg) \d \mu^{-}(s)\nonumber \\
& -\int_{\Omega}f(x,u_n) v \d x \Bigg]=0. 
\end{align}
We will prove this lemma by establishing the following two claims.\\

\noindent \textbf{Claim 1:} First, we will claim that $(u_n)$ is bounded in $X_p(\Omega)$.\\
Using \eqref{level c} and \eqref{derivative} with the test function $v=\frac{u_n}{[\eta_p(u_n)]}$ for any $n\in\mathbb{N}$, there exists a constant $C>0$ such that\begin{align}\label{eq3.3}    |\mathcal{I}(u_n)|\leq C,~\text{and}\nonumber\\    \bigg|\langle \mathcal{I}'(u_n),\frac{u_n}{\eta_p(u_n)}\rangle \bigg|\leq C.\end{align}
From \eqref{eq3.3}, we get
\begin{align}\label{eq3.4}
    \mathcal{I}(u_n)-\frac{1}{\mu}\langle \mathcal{I}^\prime(u_n),u_n\rangle\leq C\bigg(1+\frac{\eta_p(u_n)}{\mu}\bigg)\leq C\left(1+\eta_p(u_n)\right),
\end{align}
where $\mu$ is the constant in \eqref{f4}. On integrating both sides of \eqref{f2}, we derive
\begin{equation}\label{eq3.5}
    |F(x,t)|\leq a_1|t|+\frac{a_2}{q}|t|^q,~\forall~ x\in \bar{\Omega} ~\text{and}~\forall ~x\in\mathbb{R}.
\end{equation}
Using \eqref{f2} and \eqref{eq3.5}, we obtain
\begin{align}\label{eq3.6}
   & \bigg| \int_{\Omega \cap \{|u_n|\leq r\}} \left(F(x,u_n(x))-\frac{1}{\mu}f(x,u_n(x))u_n\right) \d x \bigg|\nonumber\\ 
    \leq & \int_{\Omega \cap \{|u_n|\leq r\}} \bigg( a_1 |u_n|+\frac{a_2}{q}|u_n|^q +\frac{a_1}{\mu}|u_n| +\frac{a_2}{\mu} |u_n|^{q}\bigg) \d x \nonumber\\
    \leq & \bigg( a_1 r+\frac{a_2}{q}r^q +\frac{a_1}{\mu}r +\frac{a_2}{\mu} r^{q}\bigg)|\Omega|:=C_1,
\end{align}
where $\mu$ is the constant in \eqref{f4}. Consequently, applying \eqref{f4} and \eqref{eq3.6}, we derive
\begin{align}\label{eq3.7}
&\mathcal{I}(u_n)-\frac{1}{\mu}\langle \mathcal{I}^\prime(u_n),u_n\rangle \nonumber\\
=& \left( \frac{1}{p}-\frac{1}{\mu}\right)[\eta_p(u_n)]^p-\int_\Omega \bigg(  F(x,u_n(x))-\frac{1}{\mu}f(x,u_n(x))u_n\bigg) \d x  \nonumber \\
\geq &\left( \frac{1}{p}-\frac{1}{\mu}\right)[\eta_p(u_n)]^p- \int_{\Omega \cap \{|u_n|\leq r\}} \left(F(x,u_n(x))-\frac{1}{\mu}f(x,u_n(x))u_n\right) \d x \nonumber \\
\geq & \left( \frac{1}{p}-\frac{1}{\mu}\right)[\eta_p(u_n)]^p-C_1.
\end{align}
Therefore, using \eqref{eq3.4} and \eqref{eq3.7}, we get
\begin{align}\label{eq3.8}
    \left( \frac{1}{p}-\frac{1}{\mu}\right)[\eta_p(u_n)]^p-C_1 \leq C\bigg(1+\frac{[\eta_p(u_n)]}{\mu}\bigg).
\end{align}
Therefore, \eqref{eq3.8} gives us $\eta_{p}\left(u_{n}\right)$ is bounded for all $n\in \mathbb{N}$ using the fact that $\mu>p$.\\

\textbf{Claim 2:} Next, we claim that $u_{n} \rightarrow u$ strongly in $X_{p}(\Omega)$ up to a subsequence.\\
Applying Lemma \ref{Uniform convexity} and Proposition \ref{compact and cont embedding}, there exists $u \in X_{p}(\Omega)$ and a subsequence of $(u_n)$, still denoted by $(u_n)$ such that
\begin{align}\label{convergences}
\begin{split}
    & u_{n} \rightharpoonup u \text { in } X_{p}(\Omega), \\
& u_{n} \rightarrow u \text { in } L^{r}(\Omega) \text { for any } r \in [1, p_{s_{\sharp}}^{*}),  \\
& u_{n} \rightarrow u \text { a.e. in } \Omega,
\end{split}
\end{align}
that is, for each $v\in X_p(\Omega)$, we get
\begin{align}\label{eq3.10}
    \lim _{n \rightarrow\infty} &\int_{[0,1]}\left(\iint_{\mathbb{R}^{2 N}} \frac{\left|u_{n}(x)-u_{n}(y)\right|^{p-2}\left(u_{n}(x)-u_{n}(y)\right)(v(x)-v(y))}{|x-y|^{N+s p}} \d x \d y\right) \d \mu^{ \pm}(s)\nonumber \\
&=\int_{[0,1]}\left(\iint_{\mathbb{R}^{2 N}} \frac{|u(x)-u(y)|^{p-2}(u(x)-u(y))(v(x)-v(y))}{|x-y|^{N+s p}} \d x \d y\right) \d \mu^{ \pm}(s).
\end{align}
Moreover, using \eqref{convergences}, there exists $ l \in L^r(\Omega) $ such that
\begin{align}\label{eq3.111}
    |u_n(x)| \leq l(x),~a.e~\text{in}~\Omega,~\forall~n\in \mathbb{N}.
\end{align}
  Again, using \eqref{f1} \eqref{f2}, \eqref{convergences}, \eqref{eq3.10}, \eqref{eq3.111} and Dominated Converges Theorem, we obtain
 \begin{align}\label{eq3.11}
     \lim_{n\rightarrow \infty}\int_\Omega f(x,u_n)u_n \d x = \int_\Omega f(x,u)u \d x
 \end{align}
 and 
  \begin{align}\label{eq3.12}
     \lim_{n\rightarrow \infty}\int_\Omega f(x,u_n)u \d x = \int_\Omega f(x,u)u \d x.
 \end{align}
 Now substituting $v:= \pm u_{n}$ into \eqref{derivative}, we get
\begin{align}\label{new sg}
     \lim _{n \rightarrow\infty} &\langle \mathcal{I}^{\prime}\left(u_{n}\right), u_n\rangle\nonumber \\
     &=\lim _{n \rightarrow\infty} \left(\left[\rho_{p}(u_n)\right]^{p}-\int_{[0, \bar{s}]}[u_n]_{s, p}^{p} \d \mu^{-}(s)-\int_\Omega f(x,u_n)u_n\right)=0.
     \end{align}
Consequently, by employing \eqref{eq3.11} and \eqref{new sg}, we deduce
 \begin{align}\label{eq3.15}
     \lim _{n \rightarrow\infty} \left(\left[\rho_{p}(u_n)\right]^{p}-\int_{[0, \bar{s}]}[u_n]_{s, p}^{p} \d \mu^{-}(s)\right)=\int_\Omega f(x,u)u.
 \end{align}
 Again, substituting $v:= \pm u$ into \eqref{derivative}, we get
 \begin{align}\label{eq3.16}
     \lim _{n \rightarrow\infty} &\langle \mathcal{I}^{\prime}\left(u_{n}\right), u\rangle\nonumber \\
     =& \lim_{n\rightarrow \infty}\int_{[0,1]}\left(\iint_{\mathbb{R}^{2 N}}  \frac{C_{N, s, p}|u_n(x)-u_n(y)|^{p-2}(u_n(x)-u_n(y))(u(x)-u(y))}{|x-y|^{N+s p}} \d x \d y\right) \d \mu^{+}(s) \nonumber \\
& -\int_{[0, \bar{s}]}\left(\iint_{\mathbb{R}^{2 N}} \frac{C_{N, s, p}|u_n(x)-u_n(y)|^{p-2}(u_n(x)-u_n(y))(u(x)-u(y))}{|x-y|^{N+s p}} \d x \d y\right) \d \mu^{-}(s) \nonumber \\
& -\int_{\Omega}f(x,u_n) u \d x=0.
 \end{align}
 Consequently, by using \eqref{eq3.10} with $v=u$ and \eqref{eq3.12} in \eqref{eq3.16}, we deduce
 \begin{align}\label{eq3.17}
\lim _{n \rightarrow\infty} \left(\left[\rho_{p}(u)\right]^{p}-\int_{[0, \bar{s}]}[u]_{s, p}^{p} \d \mu^{-}(s)\right)=\int_\Omega f(x,u)u.
 \end{align}
Therefore, using \eqref{eq3.15} and \eqref{eq3.17}, we conclude 
\begin{align}\label{eq3.18}
    \lim_{n\rightarrow \infty}[\eta_p(u_n)]=[\eta_p(u)].
\end{align}
 Thanks to Lemma \ref{B-L lemma} and \eqref{eq3.18}, we obtain
\begin{align}\label{eq3.21}
    \int_{[0,1]}\left[\eta(u_{n}-u)\right]^{p} \d \mu^{ +}(s)=0.
\end{align}
This completes the proof.
\end{proof}

Now, to deal with the superlinear nonlinearity without the AR condition but satisfying \eqref{f6} and \eqref{f7}, we will prove the Cerami condition for the energy functional $\mathcal{I}$. Prior to establishing this, note that, by employing \eqref{f3} and \eqref{f7}, there exists a constant $\alpha \geq 1$ such that
\begin{align}\label{eq3.22} 
  \mathcal{F}(x,t') \leq \alpha \mathcal{F}(x,t), \text{ for all } x \in \Omega  \text{ and } t, t' \in \mathbb{R} \text{ with } 0 < |t'| \leq |t|. 
  \end{align}

\begin{lem}\label{lmn3.2}
       Assume that $f:\bar{\Omega}\times \mathbb{R}\rightarrow \mathbb{R}$ satisfies the conditions \eqref{f1}, \eqref{f2}, \eqref{f3}, \eqref{f6} and \eqref{f7}. Then $\mathcal{I}$ holds the Cerami condition $\mathrm{(Ce)}$ at any level $c\in \mathbb{R}$.
\end{lem}
\begin{proof}
Let $(u_{n}) \subset X_{p}(\Omega)$ be a $\mathrm{Ce}$ sequence at $c\in\mathbb{R}$ for the energy functional $\mathcal{I}$. Then, we get 
   
\begin{align}\label{c2}
\begin{split}
    \lim _{n \rightarrow\infty} \mathcal{I}\left(u_{n}\right) =&\lim _{n \rightarrow\infty} \Bigg[\frac{1}{p}\left[\rho_{p}\left(u_{n}\right)\right]^{p}-\frac{1}{p} \int_{[0, \bar{s}]}\left[u_{n}\right]_{s, p}^{p} \d \mu^{-}(s) -\int_\Omega\int_0^{u_n}f(x,\tau)\d \tau \d x \Bigg]=c
\end{split}
\end{align}
and
\begin{align}\label{d2}
& \lim _{n \rightarrow\infty} \left(1+[\eta_p(u_n)]\right) \sup_{v \in X_p(\Omega)} \langle \mathcal{I}^{\prime}\left(u_{n}\right), v\rangle = \lim _{n \rightarrow\infty} \left(1+[\eta_p(u_n)]\right) \times \nonumber\\
& \sup_{v \in X_p(\Omega)} \Bigg[ \int_{[0,1]}\Bigg( \iint_{\mathbb{R}^{2 N}}  \frac{\splitfrac{C_{N, s, p}\left|u_{n}(x)-u_{n}(y)\right|^{p-2}\left(u_{n}(x)-u_{n}(y)\right)}{\times(v(x)-v(y))}}{|x-y|^{N+s p}}\d x \d y\Bigg) \d \mu^{+}(s) \nonumber\\ 
& -\int_{[0, \bar{s}]}\Bigg( \iint_{\mathbb{R}^{2 N}} \frac{\splitfrac{C_{N, s, p}\left|u_{n}(x)-u_{n}(y)\right|^{p-2}\left(u_{n}(x)-u_{n}(y)\right)}{\times(v(x)-v(y))}}{|x-y|^{N+s p}} \d x \d y\Bigg) \d \mu^{-}(s)\nonumber \\
& -\int_{\Omega}f(x,u_n) v \d x \Bigg]=0. 
\end{align}
We will prove it in two steps.\\

\textbf{Step 1:-} At first, we claim that $(u_n)$ is bounded in $X_p(\Omega)$. We will show it by the method of contradiction.  Let $(u_n)$ be unbounded in $X_p(\Omega)$; then there exists a subsequence (also denoted by $u_n$) such that 
\begin{align}\label{eq3.25} 
\lim_{n\rightarrow \infty} [\eta_p(u_n)] = \infty. 
\end{align}
Using \eqref{d2} and \eqref{eq3.25}, we obtain
\begin{align}\label{eq3.26}
    \lim _{n \rightarrow\infty} \sup_{v \in X_p(\Omega)} \langle \mathcal{I}^{\prime}\left(u_{n}\right), v\rangle=0.
\end{align}
Now for each $n\in \mathbb{N}$, we define 
\begin{align}\label{eq3}
    v_n=\frac{u_n}{[\eta_p(u_n)]}.
    \end{align}
    Clearly, $v_n$ is bounded in $X_p(\Omega)$. Thus, using Lemma \ref{Uniform convexity} and Proposition \ref{compact and cont embedding}, there exists $v \in X_{p}(\Omega)$ and a subsequence of $(v_n)$, still denoted by $(v_n)$ such that
\begin{align}\label{eq3.27}
\begin{split}
    & v_{n} \rightharpoonup v \text { in } X_{p}(\Omega) \\
& v_{n} \rightarrow v \text { in } L^{r}(\Omega) \text { for any } r \in [1, p_{s_{\sharp}}^{*}),  \\
& v_{n} \rightarrow v \text { a.e. in } \Omega.
\end{split}
\end{align}
Moreover, using \eqref{eq3.27} there exists $l\in L^r(\Omega)$ such that \begin{align}\label{eq3.28}
    |v_n(x)|\leq l(x)\text{ a.e. in }~\Omega,~ \forall ~n\in \mathbb{N}.
\end{align} 

 We will consider two cases: first, when $v = 0$, and second, when $v \neq 0$. In both cases, we will derive a contradiction. We begin with the following scenario.\\

\noindent \textbf{ Case I (when $v=0$):} First, we suppose that $ v = 0 $. For any $n\in \mathbb{N}$, there exists $t_n\in[0,1]$ such that
\begin{align}\label{eq3.29}
    \mathcal{I}(t_nu_n)=\max_{t\in[0,1]} \mathcal{I}(tu_n).
\end{align}
For any $m\in \mathbb{N}$, there exists $r_m$ (we choose $r_m={(2pm)}^\frac{1}{p}$) such that
\begin{align}\label{eq3.30}
    \frac{r_m}{[\eta_p(u_n)]} \in (0,1),~\text{ for large } n,~ (\text{where } n>\bar{n}(p,m)).
\end{align}
Again, for each $ m \in \mathbb{N} $, using \eqref{eq3.27} and the continuity of $ F $, we derive
\begin{align}\label{eq3.31}
    \lim_{n\rightarrow \infty} F(x,r_mv_n(x))=F(x,r_mv(x))=0, \text{ a.e. }x\in \Omega.
\end{align}
Moreover, for every $ m, n \in \mathbb{N}$, integrating both sides of \eqref{f2} and employing \eqref{eq3.28}, we obtain
\begin{align}\label{eq3.32}
    |F(x,r_mv_n(x))|\leq & a_1|r_mv_n(x)|+\frac{a_2}{q} |r_mv_n(x)|^q \nonumber \\
     \leq &  a_1r_ml(x)+\frac{a_2}{q} (r_ml(x))^q \in L^1(\Omega), \text{ a.e. }x\in \Omega.
\end{align}
Therefore, using \eqref{eq3.31}, \eqref{eq3.32} and Dominated Convergence Theorem, we get
\begin{align}\label{eq3.33}
    \lim_{n\rightarrow \infty}\int_\Omega F(x,r_mv_n(x)) \d x =0, \text{ for each } m\in \mathbb{N}.
\end{align}
Thus, for each $m\in \mathbb{N}$, using \eqref{eq3.29}, \eqref{eq3.30} and \eqref{eq3.33}, we obtain
\begin{align}\label{eq3.34}
    \mathcal{I}(t_nu_n)\geq & \mathcal{I}\bigg(\frac{r_m}{[\eta_p(u_n)]}u_n\bigg)=\mathcal{I}(r_mv_n)\nonumber \\
    =& \frac{1}{p} [\eta_p(r_mv_n)]^p- \int_\Omega F(x,r_mv_n(x)) \d x \nonumber \\
    = & 2m - \int_\Omega F(x,r_mv_n(x)) \d x \geq m, \text{ for large } n.
\end{align}
From \eqref{eq3.34}, we get 
\begin{align}\label{eq3.35}
    \lim_{n\rightarrow \infty}\mathcal{I}(t_nu_n)=\infty.
\end{align}
Since $\mathcal{I}(0u_n)=\mathcal{I}(0)=0 $ and $\lim_{n\rightarrow \infty}\mathcal{I}(1u_n)=\lim_{n\rightarrow \infty}\mathcal{I}(u_n)=c$. Applying these facts with \eqref{eq3.29} and \eqref{eq3.35}, we can clearly see that $t_n\in(0,1)$ such that
$$\frac{\d}{\d t}\bigg|_{t=t_n} \mathcal{I}(tu_n)=0, \text{ for each } n\in \mathbb{N}.$$
Moreover, 
\begin{align}\label{eq3.36} 
\langle \mathcal{I}^\prime(t_nu_n),t_nu_n \rangle =t_n \frac{\d}{\d t}\bigg|_{t=t_n} \mathcal{I}(tu_n)=0.
\end{align}

To arrive at a contradiction, we will next  show that, for some suitable constant $C_2\geq 0$, we obtain
\begin{align}\label{eq3.37}
\limsup_{n\rightarrow \infty} \mathcal{I}(t_nu_n)\leq C_2,
\end{align}
which is a contradiction to \eqref{eq3.35}.\\
Using \eqref{eq3.22}, \eqref{c2}, \eqref{eq3.26} and \eqref{eq3.36}, we derive
\begin{align}
    \frac{1}{\alpha} \mathcal{I}(t_nu_n)=& \frac{1}{\alpha} \bigg( \mathcal{I}(t_nu_n)-\frac{1}{p} \langle \mathcal{I}^\prime(t_nu_n),t_nu_n \rangle \bigg)\nonumber \\
    =& \frac{1}{\alpha} \bigg( -\int_\Omega  F(x,t_n u_n(x)) \d x+\frac{1}{p}\int_\Omega f(x,t_n u_n(x))t_n u_n(x) \d x \bigg) \nonumber \\
    = & \frac{1}{\alpha} \int_\Omega \mathcal{F}(x,t_nu_n(x))\d x \nonumber \\
    \leq & \int_\Omega \mathcal{F}(x,u_n(x))\d x \nonumber \\
    = & \int_\Omega \bigg\{ \frac{1}{p} u_n(x) f(x, u_n(x)) -F(x,u_n(x)) \bigg\} \d x \nonumber \\
    =& \mathcal{I}(u_n) -\frac{1}{p} \langle \mathcal{I}^\prime(u_n), u_n \rangle \rightarrow c, \text{ as } n\rightarrow \infty. \nonumber
\end{align}
This implies that to \eqref{eq3.37}. Thus, we get a contradiction to \eqref{eq3.35} implying that our assumption that $(u_n)$ is unbounded in $X_p(\Omega)$ is not correct. Therefore, $(u_n)$ is bounded in $X_p(\Omega).$  \\

\noindent \textbf{ Case II (when $v\neq 0$):} Let us assume that $v\neq 0$. Then $$m(\Omega_1):=m\{x\in \Omega: v(x)\neq 0\}>0,$$ where $m$ is denoted for Lebesgue measure. Using \eqref{eq3} and \eqref{eq3.27}, we obtain\begin{align}\label{eq3.39}
    \lim_{n\rightarrow \infty} |u_n(x)|= \infty, \text{ a.e. } x\in \Omega_1.
\end{align}
Applying \eqref{c2} and \eqref{eq3.25}, we get
\begin{align}\label{eq3.40}
    \lim_{n\rightarrow \infty} \frac{\mathcal{I}(u_n)}{[\eta_p(u_n)]^p}=\lim_{n\rightarrow \infty}\bigg[\frac{1}{p}-\int_{\Omega_1}\frac{F(x,u_n(x))}{[\eta_p(u_n)]^p}\d x- \int_{\Omega\setminus \Omega_1}\frac{F(x,u_n(x))}{[\eta_p(u_n)]^p}\d x\bigg]=0.
\end{align}
Let us consider \eqref{eq3.40} into two parts separately.\\

\noindent \textbf{Part I (For the first integral of \eqref{eq3.40}):}
Using \eqref{f6}, \eqref{eq3.27}, and \eqref{eq3.39}, we derive
\begin{align}
    \lim_{n\rightarrow \infty} \frac{F(x,u_n(x))}{[\eta_p(u_n)]^p}\d x=& \lim_{n\rightarrow \infty} \frac{F(x,u_n(x))}{|u_n|^p} \frac{|u_n|^p}{[\eta_p(u_n)]^p} \nonumber \\
    =&\lim_{n\rightarrow \infty} \frac{F(x,u_n(x))}{|u_n|^p} |v_n(x)|^p=\infty, \text{ a.e. for any } x\in \Omega_1. \nonumber
\end{align}
Hence, by Fatou's Lemma, we obtain 
\begin{align}\label{eq3.41}
 \lim_{n\rightarrow \infty} \int_{\Omega_1}\frac{F(x,u_n(x))}{[\eta_p(u_n)]^p}\d x =\infty.
\end{align}

\noindent \textbf{Part II (For the second integral of \eqref{eq3.40}):} We will claim that there exists a positive constant, $C_3$, such that
\begin{align}\label{eq3.42}    
\int_{\Omega\setminus \Omega_1}\frac{F(x,u_n(x))} {[\eta_p(u_n)] ^p}\d x \geq -\frac{C_3}{[\eta_p(u_n)]^p}|\Omega\setminus \Omega_1|,
\end{align}
that is, using \eqref{eq3.25}, \eqref{eq3.40}, \eqref{eq3.41} and \eqref{eq3.42}, we obtain
\begin{align}\label{eq3.47}
    \lim_{n \rightarrow \infty} \int_{\Omega\setminus \Omega_1}\frac{F(x,u_n(x))}{[\eta_p(u_n)]^p}\d x \geq 0.
\end{align}
By \eqref{f6}, we obtain 
\begin{align}\label{eq3.43}
    \lim_{|t|\rightarrow \infty} F(x,t)=\infty, \text{ uniformly for each } x\in \bar{\Omega}.
\end{align}
From \eqref{eq3.43}, there exist positive constants $t_1$ and $C_4$ such that
\begin{align}\label{eq3.44}
    F(x,t)\geq C_4,~ \forall~x\in\bar{\Omega} \text{ and } |t|>t_1.
\end{align}
We know, $F$ is continuous in $\bar{\Omega}\times \mathbb{R}$, then
\begin{align}\label{eq3.45}
    F(x,t) \geq \min_{(x,t)\in \bar{\Omega}\times [-t_1,t_1]} F(x,t),~ \forall~x\in\bar{\Omega} \text{ and } |t|\leq t_1.
\end{align}
Note that $\min_{(x,t)\in \bar{\Omega}\times [-t_1,t_1]} F(x,t)\leq 0$, since $F(x,0)=0, \forall~ x\in \bar{\Omega}$.
Thus, using \eqref{eq3.44} and \eqref{eq3.45}, there exists $C_3\geq 0$ such that
\begin{align}\label{eq3.46}
    F(x,t)\geq -C_3, \text{ for each }(x,t)\in \bar{\Omega}\times \mathbb{R}. 
\end{align}
Applying \eqref{eq3.46}, we conclude that \eqref{eq3.42}.
 So, using \eqref{eq3.41} and \eqref{eq3.47} in \eqref{eq3.40} yields a contradiction.  Therefore  the sequence $(u_n)$ is bounded in $X_p(\Omega)$ in this case as well.\\
 
\textbf{Step 2:} Employing a similar argument from \textbf{Claim 2} in Lemma \ref{lmn3.1}, we deduce that $u_{n} \rightarrow u$ strongly in $X_{p}(\Omega)$, up to a subsequence. This completes the proof.
\end{proof}
\begin{rem}
    Note that the condition \eqref{f7} was essential for establishing the inequality \eqref{eq3.37} in the proof of Lemma \ref{lmn3.2}.
\end{rem}
To show the Cerami condition in the framework of Theorem \ref{t3}, we require the following significant result from \cite{L2010}.
\begin{lem} \cite[Lemma 2.3]{L2010}\label{lmn3.4}
Assume that \eqref{f9} holds. Then the function $\mathcal{F}(x,t)$ is increasing in $t\geq t''$ and decreasing in $t\leq -t''$ for all $x\in \Omega$, where $\mathcal{F}=\frac{1}{p}tf(x,t)-F(x,t)$. In particular, for any $x\in\Omega$, there exists $C_5>0$ such that 
    \begin{align}\nonumber
        \mathcal{F}(x,s)\leq \mathcal{F}(x,t)+C_5, \text{ if } ~0\leq s\leq t \text{ or } t\leq s \leq 0.
    \end{align} 
\end{lem}
\begin{lem}\label{lmn3.5}
    Suppose that the function $f:\bar{\Omega}\times \mathbb{R}\rightarrow \mathbb{R}$ satisfies the conditions \eqref{f1}, \eqref{f2}, \eqref{f6}, and \eqref{f9}. Then, the energy functional $\mathcal{I}$ satisfies the Cerami condition $\mathrm{(Ce)}$ at any level $c\in \mathbb{R}$.
\end{lem}
\begin{proof}
    This proof is also based on a similar argument in Lemma \ref{lmn3.2}, with the exception of the proof of inequality \eqref{eq3.37}. Here the condition \eqref{f7} is omitted; so, we will show the inequality \eqref{eq3.37} utilizing \eqref{f9} and Lemma \ref{lmn3.4}. In this proof, we will use the same notation as in Lemma \ref{lmn3.2}. Applying Lemma \ref{lmn3.4}, we obtain
\begin{align}\label{eq3.48}
         \lim_{n\rightarrow \infty} \mathcal{I}(t_nu_n)=&  \lim_{n\rightarrow \infty} \bigg[ \mathcal{I}(t_nu_n)-\frac{1}{p} \langle \mathcal{I}^\prime(t_nu_n),t_nu_n \rangle \bigg]\nonumber \\
    =&  \lim_{n\rightarrow \infty} \bigg[-\int_\Omega  F(x,t_n u_n(x)) \d x+\frac{1}{p}\int_\Omega f(x,t_n u_n(x))t_n u_n(x) \d x \bigg] \nonumber \\
    = & \lim_{n\rightarrow \infty} \bigg[ \int_\Omega \mathcal{F}(x,t_nu_n(x))\d x \bigg] \nonumber \\
    = & \lim_{n\rightarrow \infty} \bigg[ \int_{\{u_n\geq 0\}} \mathcal{F}(x,t_nu_n(x))\d x +\int_{\{u_n<0\}} \mathcal{F}(x,t_nu_n(x))\d x \bigg] \nonumber \\
    \leq & \lim_{n\rightarrow \infty} \bigg[ \int_{\{u_n\geq 0\}}\bigg( \mathcal{F}(x,u_n(x))+C_5\bigg) \d x +\int_{\{u_n<0\}} \bigg( \mathcal{F}(x,t_nu_n(x))+C_5 \bigg) \d x \bigg] \nonumber \\
    = & \lim_{n\rightarrow \infty} \bigg[\int_\Omega \mathcal{F}(x,u_n(x))\d x + C_5|\Omega| \bigg] \nonumber \\
    = & \lim_{n\rightarrow \infty} \bigg[ \mathcal{I}(u_n)-\frac{1}{p} \langle \mathcal{I}^\prime(u_n),u_n \rangle + C_5|\Omega| \bigg]
    = c+ C_5|\Omega| .
    \end{align}
    Therefore, \eqref{eq3.48} implies the inequality \eqref{eq3.37}. Hence $\mathcal{I}$ satisfies the Cerami condition. This completes the proof.
\end{proof}

\section{Fountain Theorem and A Significant result} \label{s4}
In this section, we will state the Fountain Theorem and establish an interesting lemma. We follow Bartsh's Fountain Theorem \cite{B1993} regarding the energy functional $\mathcal{I}$. We adopt the notation from \cite[Theorem 2.5]{B1993} ( also refer to \cite[Theorem 2.9]{L2010} and \cite[Theorem 3.6]{W1996}),  specifically for each $k\in \mathbb{N}$ and $1\leq q < p^*_{s_\sharp}$.
  Since $X_p(\Omega)$ is a separable and reflexive Banach space, then there exists a sequence $(v_m)\subset X_p(\Omega)$ and $(\phi_n)\subset {X_p(\Omega)}^*$ such that
 \begin{itemize}
      \item[$(1)$] $\langle \phi_n,v_m \rangle$ 
      $=\begin{cases}  
        1 \text{ if } m=n, \\
         0 \text{ if } m\neq n, 
     \end{cases}$
     \item[$(2)$] $\overline{\text{span}\{v_m:m\in\mathbb{N}\}}=X_p(\Omega)$ and $\overline{\text{span}\{\phi_n:n\in\mathbb{N}\}}={X_p(\Omega)}^*$.\\ 
     \noindent Let
     $X_j=\mathbb{R}v_j$, then $X_p(\Omega)=\overline{\bigoplus_{j\geq 1}X_j}$. Now we will define the spaces
     $$A_k=\bigoplus_{j=1}^k X_j,\,\,\, \text{and}\,\,\,\quad B_k=\overline{\bigoplus_{j\geq k}X_j}.$$
 \end{itemize}
 
Since $A_k$ is finite-dimensional, all norms on $A_k$ are equivalent. Therefore, for any $u\in A_k$, then there exist two positive constants $D_{k,q}$ and $D'_{k,q}$ such that
\begin{align}\label{eq4.1}
    D_{k,q} [\eta_p(u)]\leq ||u||_{L^q(\Omega)} \leq D'_{k,q} [\eta_p(u)].
    \end{align}
\textbf{Fountain Theorem:} Let $\mathcal{I}$ be a $C^1$ functional on an infinite-dimensional Banach space $X$, and
\begin{enumerate}    
\item[(i)] $\mathcal{I}$ holds Palais-Smale condition,   
\item [(ii)] $\mathcal{I}$ is even, and    
\item [(iii)] $\mathcal{I}$ satisfies the geometric assumption, which is defined as follows: for every $k\in \mathbb{N}$,
\begin{enumerate}    
\item [(a)] there exists $a'_k>0$ such that $y_k:=\max\left\{\mathcal{I}(u):u \in A_k, ~[\eta_n(u)] =a'_k\right\}\leq 0,$   
\item [(b)] there exists $b_k>0$ such that $z_k:=\inf\left\{\mathcal{I}(u):u \in B_k,~ [\eta_p(u)] =b_k\right\}\rightarrow \infty$ as $k\rightarrow \infty$, with $0<b_k <a'_k$,
\end{enumerate}
\end{enumerate}
then there exists an unbounded sequence of critical values $u_k$ such that $\mathcal{I}(u_n)\rightarrow \infty$.\\

Before going to the proof of our main results in the next section, first, we will prove a lemma as follows:
\begin{lem}\label{lmn4.1}
    Let $1\leq q < p^*_{s_\sharp}$, $k\in \mathbb{N}$ and 
    $$\beta_k:=\sup \left\{||u||_{L^q(\Omega)}: u\in B_k,~ [\eta_p(u)]=1 \right\},$$
    Then $\beta_k\rightarrow 0$ as $k\rightarrow \infty$.
\end{lem}
\begin{proof}
    From the Definition of $B_k$, we see $B_{k+1}\subset B_k$, then $0<\beta_{k+1}\leq \beta_k$ for any $k\in \mathbb{N}$. Hence, there exists $\beta\geq 0$ such that 
    \begin{align}\label{eq4.2}
        \beta_k \rightarrow \beta, \text{ as } k\rightarrow \infty.
    \end{align}
    Again, by Definition of $\beta_k$, for any $k\in \mathbb{N}$, there exists $u_k\in B_k$ such that
    \begin{align}\label{eq4.3}
        [\eta_p(u_k)]=1\text{  and  } ||u_k||_{L^q(\Omega)}>\frac{\beta_k}{2}.
    \end{align}
    Since $X_p(\Omega)$ is a uniform convex Banach space, it is reflexive. Then using \eqref{eq4.3} and Lemma \ref{weak convergence}, there exists a subsequence $u_k$ (same denoted by $u_k$) such that $u_k \rightharpoonup u \in X_p(\Omega)$ and for any $v\in X_p(\Omega)$, we get
\begin{align}\label{eq4.4}
    \lim _{k \rightarrow\infty} &\int_{[0,1]}\left(\iint_{\mathbb{R}^{2 N}} \frac{\left|u_{k}(x)-u_{k}(y)\right|^{p-2}\left(u_{k}(x)-u_{k}(y)\right)(v(x)-v(y))}{|x-y|^{N+s p}} \d x \d y\right) \d \mu^{ \pm}(s)\nonumber \\
&=\int_{[0,1]}\left(\iint_{\mathbb{R}^{2 N}} \frac{|u(x)-u(y)|^{p-2}(u(x)-u(y))(v(x)-v(y))}{|x-y|^{N+s p}} \d x \d y\right) \d \mu^{ \pm}(s).
\end{align}
Since $B_k$ is convex and closed, thus it is closed for the weak topology. Hence
$$u\in \cap_{k=1}^\infty B_k=\{0\}.$$
Therefore, by Sobolev embedding proposition \ref{compact and cont embedding}, as $k\rightarrow \infty$, we obtain
\begin{align}\label{eq4.5}
    u_k\rightarrow 0 \text{ in } L^q(\Omega).
\end{align}
Using \eqref{eq4.2}, \eqref{eq4.3} and \eqref{eq4.5}, we conclude that $\beta_k \rightarrow 0$ as $k\rightarrow \infty$. This completes the proof.  
\end{proof}

\section{Proof of main results} \label{s5}
In this section, we present the proof of the main results. We begin with the proof of Theorem \ref{t1}
\begin{proof}[Proof of Theorem \ref{t1}] \label{proof of t1}
    We follow the proof given in \cite[Theorem 3.7]{W1996}. It is established that $\mathcal{I}$ fulfils the Palais-Smale condition as shown by Lemma \ref{lmn3.1}. Moreover, using \eqref{f3}, we obtain $\mathcal{I}(-u)=\mathcal{I}(u)$ for $u\in X_p(\Omega)$. To apply the Fountain theorem, it is necessary to satisfy only the geometric conditions of the energy functional $\mathcal{I}$. We will prove this in two claims.\\
    \textbf{Claim I:} For each $k\in \mathbb{N}$, there exists $a'_k>0$ such that $$y_k=\max\left\{\mathcal{I}(u):u \in A_k, ~[\eta_p(u)]=a'_k\right\}\leq 0.$$
For $u\in A_k$, using \eqref{f5} and \eqref{eq4.1}, we obtain
\begin{align}\label{eq5.1}
    \mathcal{I}(u)\leq & \frac{1}{p} [\eta_p(u)]^p -a_3||u||^\mu_{L^\mu(\Omega)}+a_4|\Omega| \nonumber \\
    \leq & \frac{1}{p} [\eta_p(u)]^p - D''_{k,\mu}[\eta_p(u)]^\mu+a_4|\Omega|, \text{ for some suitable positive constant } D''_{k,\mu}.
\end{align}
Since, $\mu > p$.  Utilizing \eqref{eq5.1}, for any $u\in A_k$ and a sufficiently large $a'_k>0$ such that $[\eta_p(u)]=a'_k$, we conclude that 
$$\mathcal{I}(u)\leq 0.$$

\textbf{Claim II:}  For each $k\in \mathbb{N}$, there exists $b_k>0$ such that $$z_k=\inf\left\{\mathcal{I}(u):u \in B_k,~ [\eta_p(u)]=b_k\right\}\rightarrow \infty \text{ as } k\rightarrow \infty.$$
    From \eqref{eq3.5}, there exists a constant $C>0$ such that
    \begin{align}\label{eq5.2}
        |F(x,t)|\leq C(1+|t|^q),~\forall~ x\in \bar{\Omega} ~\text{and}~\forall ~x\in\mathbb{R}.
    \end{align}
    Thus, using \eqref{eq5.2} and for any $u\in B_k\setminus \{0\}$, we get
    \begin{align}\label{eq5.3}
        \mathcal{I}(u)\geq & \frac{1}{p} [\eta_p(u)]^p -C||u||^q_{L^q(\Omega)}-C|\Omega| \nonumber \\
        = & \frac{1}{p} [\eta_p(u)]^p -C\left|\left|\frac{u}{[\eta_p(u)]}\right|\right|^q_{L^q(\Omega)}[\eta_p(u)]^q-C|\Omega| \nonumber \\
        \geq & \frac{1}{p} [\eta_p(u)]^p -C \beta_k^q[\eta_p(u)]^q-C|\Omega| \nonumber \\
        = & [\eta_p(u)]^p \bigg( \frac{1}{p}- C \beta_k^q[\eta_p(u)]^{q-p} \bigg)-C|\Omega|, ~\beta_k\text{ is defined in Lemma \ref{lmn4.1}}.
    \end{align}
    Let us select $ b_k = (qC\beta^q_k)^{-\frac{1}{q-p}} $; consequently, $b_k \rightarrow \infty $ as $ k \rightarrow \infty $ due to Lemma \ref{lmn4.1} and the condition $ q > p $. Therefore, using \eqref{eq5.3} and for each $u\in B_k$ with $[\eta_p(u)]=b_k$, we obtain
    \begin{align}\nonumber
         \mathcal{I}(u)\geq b_k^p\bigg( \frac{1}{p}-\frac{1}{q} \bigg) -C|\Omega|\rightarrow \infty \text{ as } k\rightarrow \infty.
    \end{align}
    This conclude the proof of this result.
\end{proof}
\begin{rem}
    We observe that we used only the Ambrosetti-Rabinowitz (AR) condition \eqref{f4} (actually we use \eqref{f5}) for showing the \textbf{Claim I} in the proof of geometric conditions of the energy functional $\mathcal{I}$. Moreover, the proof of \textbf{Claim II} in geometric conditions of the energy functional $\mathcal{I}$ uses \eqref{f2} and the Sobolev embedding proposition \ref{compact and cont embedding}.
\end{rem}

Now, we proceed to prove our next main result, that is, Theorem \ref{t2}.
\begin{proof}[Proof of Theorem \ref{t2}] 
    Since $\mathcal{I}$ satisfies the Cerami condition (Lemma \ref{lmn3.2}), it follows that $\mathcal{I}$ also fulfills the Palais-Smale condition. Additionally, using \eqref{f3}, we obtain $\mathcal{I}(-u)=\mathcal{I}(u)$ for $u\in X_p(\Omega)$, and the verification of the geometric condition $(b)$ for the energy functional $\mathcal{I}$ is shown in \textbf{Claim II} of Theorem \ref{t1}. To apply the Fountain theorem, we will verify only the geometric condition $(a)$ for the energy functional $\mathcal{I}$. To show this, we will utilize the linear subspace $A_k$ and \eqref{f6}.\\
    \textbf{Claim I:} Using \eqref{f6}, for each $k\in \mathbb{N}$, there exists $r_k>0$ such that 
    \begin{align}\label{eq6.1}
        F(x,t)\geq \frac{1}{D_{k,p}^p}|t|^p, ~\forall ~x\in\bar{\Omega}, ~\forall~t\in \mathbb{R} \text{ and } \forall~ |t|>r_k,
    \end{align}
    where $D_{k,p}$ is the positive constant given in \eqref{eq4.1} when $q=p$. Moreover, by continuity of $F$, we obtain
    \begin{align}\label{eq6.2}
        F(x,t) \geq m_k:= \min_{x\in \bar{\Omega},|t|\leq r_k} F(x,t).
    \end{align}
    Since, $F(x,0)=0$ for each $x\in \bar{\Omega}$, then $m_k\leq 0$. Using \eqref{eq6.1} and \eqref{eq6.2}, we get
\begin{align}\label{eq6.3}
        F(x,t)\geq \frac{1}{D_{k,p}^p}|t|^p-Z_k, ~\forall ~x\in\bar{\Omega}, ~\forall~t\in \mathbb{R},
    \end{align}
    where for some positive constant $Z_k$ (say, $Z_k \geq \frac{r_k^p}{D_{k,p}^p}-m_k$). Using \eqref{eq4.1}, \eqref{eq6.3} and for any $u\in A_k$, we conclude
    \begin{align}\label{eq6.4}
         \mathcal{I}(u)\leq & \frac{1}{p} [\eta_p(u)]^p - \frac{1}{D_{k,p}^p}||u||_{L^p(\Omega)}^p+Z_k|\Omega| \nonumber \\ 
         \leq & \frac{1}{p} [\eta_p(u)]^p - [\eta_p(u)]^p + Z_k|\Omega| \nonumber \\ 
         = & -\frac{(p-1)}{p}[\eta_p(u)]^p+ Z_k|\Omega|  .
    \end{align}
    Consequently, for a sufficiently large $a_k$ such that $[\eta_p(u)]=a_k$, it follows from \eqref{eq6.4} that $\mathcal{I}(u)\leq 0$. This concludes the proof of the theorem.
\end{proof}

Finally, we discuss the proof of Theorem \ref{t3}.
\begin{proof}[Proof of Theorem \ref{t3}]
   By Lemma~\ref{lmn3.5}, the functional $\mathcal{I}$ satisfies the Cerami condition, and consequently, it also fulfills the Palais--Smale condition. Moreover, for any $u \in X_p(\Omega)$, we have $
\mathcal{I}(-u) = \mathcal{I}(u),$
as a direct consequence of the symmetry assumption \eqref{f3}.  

Regarding the geometric conditions $(a)$ and $(b)$ for the energy functional $\mathcal{I}$, these follow from \textbf{Claim II} in the proof of Theorem~\ref{t1} and \textbf{Claim I} in the proof of the previous theorem. This completes the proof.

\end{proof}

\section*{Conflict of interest statement}
On behalf of all authors, the corresponding author states that there is no conflict of interest.

\section*{Data availability statement}
Data sharing is not applicable to this article as no datasets were generated or analysed during the current study.

\section*{Acknowledgement}

SB would like to thank the Council of Scientific and Industrial Research (CSIR), Govt. of India for the financial assistance to carry out this research work [grant no. 09/0874(17164)/ 2023-EMR-I]. SG acknowledges the research facilities available at the Department of Mathematics, NIT Calicut. VK is supported by the FWO Odysseus 1 grant G.0H94.18N: Analysis and Partial Differential Equations and the Methusalem program of the Ghent University Special Research Fund (BOF) (Grant number 01M01021). VK is also supported by FWO Senior Research Grant G011522N.


\end{document}